\documentclass[11pt]{amsart}
\usepackage{amsmath,amsthm,amscd,amsfonts, amssymb}
\usepackage{graphicx}
\usepackage{tikz}
\usepackage{color}
\include{diagxy}
\usepgflibrary{shapes}
\newcommand{\sect}[1]{\section{#1}\setcounter{equation}{0}}

\font\mbn=msbm10 scaled \magstep1
\font\mbs=msbm7 scaled \magstep1
\font\mbss=msbm5 scaled \magstep1
\newfam\mbff
\textfont\mbff=\mbn
\scriptfont\mbff=\mbs
\scriptscriptfont\mbff=\mbss

\newcommand{\Z}        {\mathbb{Z}  }

\newtheorem{Th}{Theorem}[section]
\newtheorem{Lm}[Th]{Lemma}
\newtheorem{C}[Th]{Corollary}

\newtheorem{Proposition}[Th]{Proposition}
\newtheorem{R}[Th]{Remark}

\newtheorem*{P1}{Problem 1}
\newtheorem*{P2}{Problem 2}
\newtheorem*{P3}{Problem 3}
\newtheorem*{P3'}{Problem 3$'$}
\newtheorem*{P4}{Problem 4}
\newtheorem*{P5}{Problem 5}
\begin{document}

\title[Universal Curves in the Center Problem for Abel Differential Equations]{Universal Curves in the Center Problem for Abel Differential Equations}

\author{Alexander Brudnyi} 
\address{Department of Mathematics and Statistics\newline
\hspace*{1em} University of Calgary\newline
\hspace*{1em} Calgary, Alberta\newline
\hspace*{1em} T2N 1N4}
\email{albru@math.ucalgary.ca}

\thanks{Research supported in part by NSERC}

\begin{abstract}
We study the center problem for the class $\mathcal E_\Gamma$ of Abel differential equations $\frac{dv}{dt}=a_1 v^2+a_2 v^3$, $a_1,a_2\in L^\infty ([0,T])$, such that images of Lipschitz  paths $\tilde A:=\bigl(\int_0^\cdot a_1(s)ds, \int_0^\cdot a_2(s)ds\bigr): [0,T]\rightarrow\mathbb R^2$ belong to a fixed compact rectifiable curve $\Gamma$. Such a curve is called universal if whenever an equation in $\mathcal E_\Gamma$ has center on $[0,T]$, this center must be universal, i.e. all iterated integrals in coefficients $a_1, a_2$ of this equation must vanish. We investigate some basic properties of universal curves. Our main results include an algebraic description of a universal curve in terms of a certain homomorphism of its fundamental group into the group of locally convergent invertible power series with product being the composition of series,
explicit examples of universal curves and approximation of Lipschitz triangulable curves by universal ones.
\end{abstract}

\date{} 

\maketitle


\sect{Introduction}
In the paper we study the center problem for the Abel differential equation
\begin{equation}\label{eq1}
\frac{dv}{dt}=a_1 v^2+a_2 v^3
\end{equation}
with coefficients $a_1,a_2$ in the space $L^\infty ([0,T])$ of bounded  Lebesgue measurable real functions on an interval $[0,T]\Subset\mathbb R$.
Recall that equation \eqref{eq1} has a {\em center} on $[0,T]$ if for all sufficiently small initial values $v_0$ the corresponding solution satisfies
$v(T)=v(0):=v_0$.  The center problem is {\em to describe explicitly the set of coefficients $a_1, a_2$ for which the corresponding equations \eqref{eq1} have centers on $[0,T]$}. The problem is closely related to the classical Poincar\'{e} Center-Focus problem for planar polynomial vector fields, see, e.g., \cite{BRY}, [I] and references therein. Recently there has been an intensive study of the center problem for Abel differential equations focused on some composition conjectures for equations with piecewise smooth coefficients, see \cite{ AL,  A1, A2, A3,  A4, B2, B6, B3, B4, B5, B7, BlRY, BFY1, BFY2, BRY, BY, C, CGM1, CGM2, CGM3, GGL, P1, P2, PRY}.

 It is known, see, e.g., \cite{B1}, that the set of centers of equation \eqref{eq1} consists of pairs $a=(a_1,a_2)$ satisfying an infinite system of equations $c_i=0$, $i\in\mathbb N$, where
\begin{equation}\label{eq2}
\begin{array}{lr}
\displaystyle
c_i:=\sum_{i_{1}+\cdots +i_{k}=i}c_{i_{1},\dots, i_{k}}(i)\cdot I_{i_{1},\dots, i_{k}}(a),\quad\text{all}\quad i_1,\dots, i_k\in\{1,2\},\\
\\
\displaystyle 
c_{i_{1},\dots, i_{k}}(i):=(i-i_{1}+1)(i-i_{1}-i_{2}+1)(i-i_1-i_2-i_3+1)\cdots 1,\\
\\
\displaystyle
I_{i_{1},\dots, i_{k}}(a):=\int\cdots\int_{0\leq s_{1}\leq\cdots\leq s_{k}\leq T}a_{i_{k}}(s_{k})\cdots a_{i_{1}}(s_{1})\ \!ds_{k}\cdots ds_{1}.\smallskip
\end{array}
\end{equation}
In general, qualitative analysis of this system is highly arduous  because of complexity of the involved equations and absence of apparent recursive relations between them.
The simplest and, in a certain statistical sense, the most frequently occurring centers, so-called {\em universal centers}, are defined by vanishing of all iterated integrals $I_{i_{1},\dots, i_{k}}(a)$. Such centers  admit an equivalent characterization in terms of the {\em tree composition condition} (see \cite[Cor.\,1.12,\, Th.\,1.14]{B2}):\smallskip

{\em Suppose that image $\Gamma_{\tilde A}$ of the map
\begin{equation}\label{tildea}
\tilde A:=(\tilde a_1, \tilde a_2): [0,T]\rightarrow\mathbb R^2,\quad\text{where}\quad \tilde a_i(t):=\int_0^t a_i(s)ds,\quad i=1,2,
\end{equation}
is Lipschitz triangulable. Then equation \eqref{eq1} has universal center on $[0,T]$ if and only if there are a finite metric tree $\mathcal T$ and continuous maps $\hat A: [0,T]\rightarrow \mathcal T$,
$\hat A(0)=\hat A(T)$, and $p:\mathcal T\rightarrow\Gamma_{\tilde A}$ such that $\tilde A=p\circ\hat A$.}\smallskip

Recall that $\Gamma\Subset\mathbb R^2$ is a Lipschitz triangulable curve if there exist a finite subset $S\subset\Gamma$  and Lipschitz and locally bi-Lipschitz arcs  $h_1,\dots, h_k: (0,1)\rightarrow \mathbb R^2$ such that
$\Gamma\setminus S=\sqcup_{1\le i\le k}\, h_i(0,1)$. 
The class of Lipschitz triangulable curves includes, in particular, compact curves in $\mathbb R^2$ admitting piecewise
 $C^1$ parametrization or images of nonconstant  analytic maps $[0,T]\rightarrow\mathbb R^2$, see, e.g., \cite[Ex.\,5.1]{BY}. \smallskip

It was proved in  \cite[Sec.\,3.1]{B5} that if $\tilde A$ is nonconstant analytic, then tree $\mathcal T$ in the above characterization of universal centers of equation \eqref{eq1} is a closed interval in  $\mathbb R$. In this case we say that $\tilde A$ satisfies a {\em composition condition}. In general,  tree composition and composition conditions are not equivalent, i.e. there exist Lipschitz maps $[0,T]\rightarrow\mathbb R^2$ which factor through continuous maps into non-interval trees but do not factor through continuous maps into intervals,  see \cite[Sec.\,3]{BY}.\smallskip

The  famous {\em composition conjecture} states that all centers of equations \eqref{eq1} with polynomial coefficients are universal (equivalently, for all centers of such equations the corresponding maps $\tilde A$ satisfy composition conditions). This conjecture is false already for Abel equations with coefficients being trigonometric polynomials, see [BFY1] for the distinction between these two cases.\smallskip

One associates with each connected Lipschitz triangulable curve $\Gamma\Subset\mathbb R^2$ containing the origin an (uncountable) family of equations \eqref{eq1} with 
$\Gamma_{\tilde A}\subset\Gamma$. Universal centers of this family admit a simple topological description. Specifically, equation \eqref{eq1} with $\Gamma_{\tilde A}\subset\Gamma$ has universal center  on $[0,T]$ if and only if {\em path $\tilde A: [0,T]\rightarrow\Gamma$ is closed and contractible in $\Gamma$}. This is equivalent to the above formulated tree composition condition for $\tilde A$, see \cite{B2} for details. At present, there is no way to explicitly describe nonuniversal centers of Abel differential equations. However, there exist Lipschitz triangulable curves $\Gamma\Subset\mathbb R^2$ satisfying the property that whenever an equation \eqref{eq1} with 
$\Gamma_{\tilde A}\subset\Gamma$ has a center on $[0,T]$, this center must be universal. We call such curves $\Gamma$ {\em universal}. The basic examples of  universal curves are (containing $0\in\mathbb R^2$)  connected {\em rectangular curves} composed of finitely many intervals each parallel to one of the coordinate axes, see \cite{B3}, or connected Lipschitz triangulable curves whose fundamental groups are either trivial or isomorphic to $\Z$, see \cite{B2}.\smallskip

In the present paper we study analytic and geometric properties of universal curves. In particular, we show that the set of ``sufficiently smooth'' universal curves  is dense with respect to the Hausdorff metric on compact subsets of $\mathbb R^2$ in the space of all connected Lipschitz triangulable curves in $\mathbb R^2$ containing the origin. \smallskip

The paper is organized as follows. In the next section we formulate the main results of the paper accompanied by some important open problems in the area. These results include an algebraic description of a universal curve in terms of a certain homomorphism of its fundamental group into the group of locally convergent invertible power series with product being the composition of series, explicit examples of universal curves and approximation of Lipschitz triangulable curves by universal ones. Section 3 is devoted to proofs of our results. 

\sect{Formulation of Main Results}
\subsection{A Characterization of Universal Curves}
We recall some results presented in \cite{B2}. Let us introduce operations $*$ and $^{-1}$ on the set $\mathcal A$ of pairs $a=(a_1, a_2)$  of coefficients of equations \eqref{eq1} by the formulas 
\[
\left((a_1,a_2)*(b_1,b_2)\right)(t):=\left\{
\begin{array}{ccc}
\left(2b_1(2t),2b_2(2t)\right)&{\rm if}&0\le t\le\frac T2\smallskip\\
\left(2a_1(2t-T),2a_2(2t-T)\right)&{\rm if}&\frac T2 \le t\le T;
\end{array}
\right.
\]
\[
\left((a_1,a_2)^{-1}\right)(t):=
\left(-a_1(T-t),-a_2(t-T)\right),\quad 0\le t\le T.
\]
It is well known that for $a,b\in\mathcal A$ iterated integrals $I_{i_{1},\dots, i_{k}}$, see \eqref{eq2}, satisfy 
\begin{equation}\label{e2.2}
I_{i_{1},\dots, i_{k}}(a*b)=
I_{i_{1},\dots, i_{k}}(a)
+\sum_{j=1}^{k-1}I_{i_{1},\dots,i_{j}}(a)\cdot I_{i_{j+1},\dots, i_{k}}(b)+I_{i_{1},\dots, i_{k}}(b).
\end{equation}
\begin{equation}\label{e2.3}
I_{i_{1},\dots, i_{k}}(a^{-1})=(-1)^{k}I_{i_{1},\dots, i_{k}}(a).
\end{equation}

For $a\in\mathcal A$ by $P(a)$ we denote the first return map of the corresponding equation \eqref{eq1} (associating with a sufficiently small initial value $r$ the value of the corresponding solution of \eqref{eq1} at $T$).
 It is represented as a convergent in a neighbourhood of $0$ power series in the initial value $r$ of the Abel equation,
\begin{equation}\label{e2.1}
P(a)(r)=r+\sum_{i=1}^\infty c_{i} r^{i+1},
\end{equation}
where $c_i$ is the weighted sum of iterated integrals in \eqref{eq2}.

By $G_c[[r]]$ we denote the group of convergent in neighbourhoods of $0$ power series of form \eqref{e2.1} with product $\circ$ defined by the composition of series. Then for $a,b\in\mathcal A$ we have
\begin{equation}\label{eq3.1}
P(a*b)=P(a)\circ P(b),\qquad P(a^{-1})=P(a)^{-1}.
\end{equation}

An important question is (cf. \cite[Question~2, page 481]{B6}):
\begin{P1}
Is it true that the first return map $P:\mathcal A\rightarrow G_c[[r]]$ is an epimorphism (i.e. each series in $G_c[[r]]$ is the first return map of an Abel equation \eqref{eq1})?
\end{P1}

Next, let $\Gamma\Subset\mathbb R^2$ be a connected  Lipschitz triangulable curve containing the origin. According to the definition, $\Gamma$ is a finite connected one-dimensional $CW$-complex and so it is homotopically equivalent to the wedge sum of finitely many circles. In particular, the fundamental group $\pi_1(\Gamma)$  of $\Gamma$ with base point $0\in\mathbb R^2$ is isomorphic, for some $m\in\mathbb Z_+$, to the free group with $m$ generators $F_m$.  (Here $F_0$ stands for the trivial group.)

By $\mathcal A_0(\Gamma)$ we denote the  subset of elements $a\in\mathcal A$ such that $\tilde A(T)=0$ and 
 $\tilde A([0,T)):=\Gamma_{\tilde A}\subset\Gamma$ (recall that  
 $\tilde A(t):=\int_0^{t}a(s)\,ds$,  $t\in [0,T]$, see \eqref{tildea}). Since $\Gamma$ admits a Lipschitz triangulation, the set $\mathcal A_0(\Gamma)$ is uncountable. (For instance, the derivative of a closed Lipschitz path passing through the origin and determined as the product, in the sense of algebraic topology, see, e.g., \cite{Hu}, of Lipschitz arcs forming the triangulation of $\Gamma$ belongs to $\mathcal A_0(\Gamma)$.)
 Clear $\mathcal A_0(\Gamma)$ is closed with respect to operations $*$ and $^{-1}$. Next, for $a\in\mathcal A_0(\Gamma)$  the closed Lipschitz path $\tilde A$ represents an element $[\tilde A]$ of $\pi_1(\Gamma)$, and the correspondence $\mathcal A_0(\Gamma)\ni a\mapsto [\tilde A]\in\pi_1(\Gamma)$ determines an epimorphism of monoids $\Psi_\Gamma: \mathcal A_0(\Gamma)\rightarrow\pi_1(\Gamma)$:
\begin{equation}\label{epi}
\Psi_\Gamma(a*b)=\Psi_\Gamma(a)\cdot\Psi_\Gamma(b),\quad \Psi_\Gamma(a^{-1})=\Psi_\Gamma(a)^{-1},\quad a,b\in \mathcal A_0(\Gamma).
\end{equation}
Since for each $\gamma\in\pi_1(\Gamma)$ the first return map $P$ is constant on $\Psi_\Gamma^{-1}(\gamma)$, see \cite{B2}, there exists a homomorphism $\hat P_\Gamma:\pi_1(\Gamma)\rightarrow G_c[[r]]$ such that
\begin{equation}\label{compos}
P=\hat P_\Gamma\circ\Psi_\Gamma\quad\text{on}\quad \mathcal A_0(\Gamma).
\end{equation}
\begin{Th}\label{te1}
Curve $\Gamma$ is universal if and only  $\hat P_\Gamma$ is a monomorphism.
\end{Th}
In particular, this implies that the group $\hat P_\Gamma(\pi_1(\Gamma))\subset G_c[[r]]$ is isomorphic to $F_m$.

\begin{P2}
Describe all possible finitely generated subgroups of $G_c[[r]]$.
\end{P2}
For instance, one can show that a ``generic'' finitely generated subgroup of $G_c[[r]]$ is free and each finitely generated solvable subgroup of $G_c[[r]]$ is isomorphic to free abelian group $\mathbb Z^m$, where $m$ can be any natural number.
However, it is even not known whether the fundamental group of a compact Riemann surface of genus $\ge 2$ can be realized as a subgroup of $G_c[[r]]$, see, e.g., \cite{B3, B4} and references therein.

\subsection{Explicit Examples of Universal Curves}
\begin{Proposition}\label{prop3.1}
Lipschitz triangulable curves $\Gamma\subset\mathbb R^2$ containing the origin whose fundamental groups $\pi_1(\Gamma)$ are either isomorphic to $F_0$ (i.e. trivial) or isomorphic to $F_1\, (\cong\mathbb Z)$  (see Figure 1 (A), (B) below) are universal.
\end{Proposition}
\begin{center}
\includegraphics[scale=1.2]{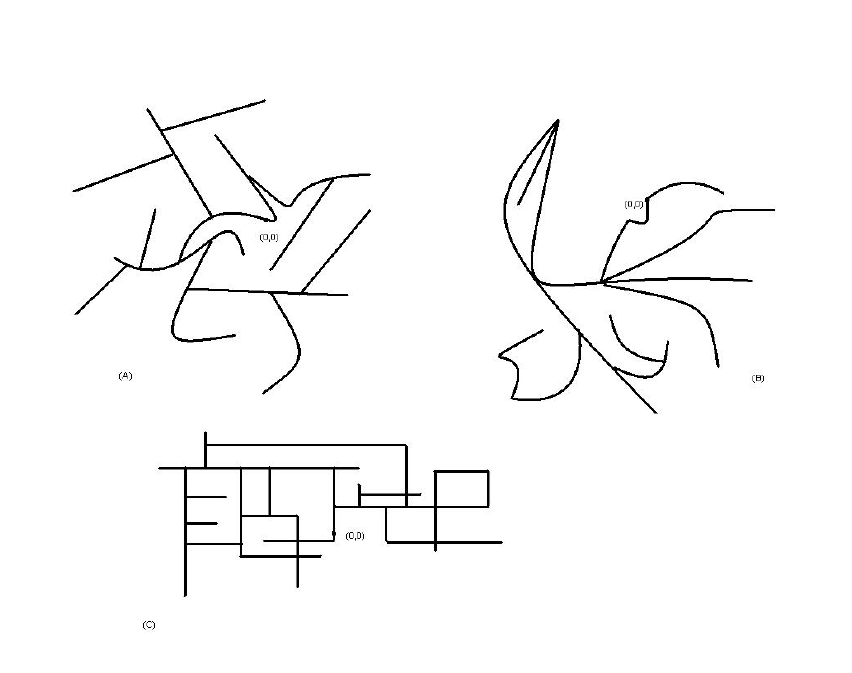} 
\end{center}
Figure 1. Examples of universal curves $\Gamma$: (A) $\pi_1(\Gamma)$ is trivial; (B) $\pi_1(\Gamma)\cong\mathbb Z$;\\  \hspace*{15mm} (C) a rectangular curve. \medskip

A curve $\Gamma\Subset\mathbb R^2$ is called {\em rectangular} if it the union of finitely many intervals each parallel to one of the coordinate axes (see Figure 1 (C)). In \cite[Th.~2.1]{B3} it was proved that {\em every connected rectangular curve $\Gamma$ containing the origin is universal}. The proof is heavily rely upon a deep result  of \cite{Co} on the structure of a certain subgroup of automorphisms of $\mathbb C$. We exploit the result of \cite{B3} in the proof of Theorem \ref{te3} but now we use its corollary \cite[Th.~2.4]{B4} to describe other explicit examples of universal curves. 

Suppose $P_1,P_2,Q_1,Q_2$ are real polynomials without constant terms such that coefficients of all of them together are algebraically independent over $\mathbb Q$. Consider polynomial paths
$A_i:=(P_i,Q_i):[0,1]\rightarrow\mathbb R^2$, $i=1,2$. Next, define the set of points
\[
S:=\{(m\cdot A_1(1),n\cdot A_2(1))\in\mathbb R^2\, :\, (m,n)\in\mathbb Z^2\}.
\]
We join each pair of points $v_1, v_2\in S$ such that $v_2-v_1=A_1(1)$ by curve $X_{v_1,v_2}:=\{v_1+A_1(t)\, :\, t\in [0,1]\}$ and each pair of points $w_1,w_2\in S$ such that $w_2-w_1=A_2(1)$ by curve $Y_{w_1,w_2}:=\{w_1+A_2(t)\, :\, t\in [0,1]\}$ thus obtaining a curvilinear lattice $\mathcal L$ in $\mathbb R^2$. Assume that polynomials $P_i,Q_i$ were chosen so that  
maps $A_i:[0,1]\rightarrow\mathbb R^2$ are embeddings and all possible curves $X_{v_1, v_2}$ and $Y_{w_1,w_2}$ (referred to as  {\em edges} of $\mathcal L$) are either disjoint or intersect by one of the points $v_1,v_2,w_1,w_2$ only (see Figure 2). Then the following result holds.
\begin{Proposition}\label{prop3.2}
Each connected curve $\Gamma\Subset\mathcal L$ containing the origin composed of finitely many edges of $\mathcal L$  is universal.
\end{Proposition}

\begin{center}
\includegraphics[scale=1.4]{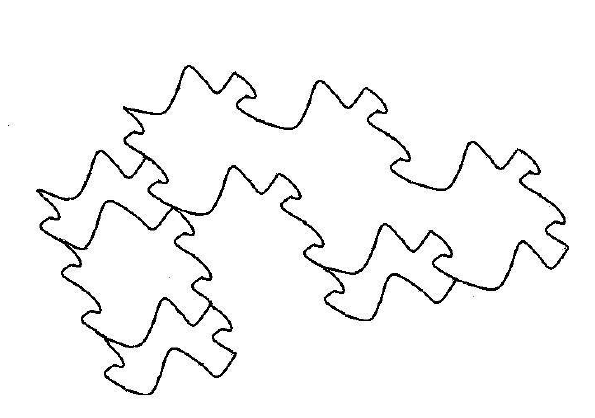}

Figure 2. An example of a curve $\Gamma$ composed of finitely many edges of a lattice $\mathcal L$. \smallskip
\end{center}

The following problem is a particular case of \cite[Question~3, page~484]{B6}:
\begin{P3}
Is it true that any connected piecewise linear curve $\Gamma\Subset\mathbb R^2$ containing the origin is universal?
\end{P3}
\begin{R}\label{rem0}
{\rm Some piecewise linear curves are universal due to Proposition \ref{prop3.2}. In general, one easily show that the first return map of the Abel equation $\frac{dv}{dt}=a v^2+bv^3$, $t\in [0,1]$, $a,b\in\mathbb R$, $|a|+|b|\ne 0$, is 
\[
P_{a,b}(r):=\left\{
\begin{array}{ccc}
\frac{1}{\psi_{a,b}^{-1}\bigl(a^2+\psi_{a,b}\bigl(\frac{1}{r}\bigr)\bigr)}&{\rm if}&a\ne 0\smallskip\\
\frac{r}{\sqrt{1-2br^2}}&{\rm if}&a=0;
\end{array}
\right|
\quad \text{here}\quad \psi_{a,b}(x):=-ax+b\ln |ax+b|,\ x\in\mathbb R.
\]
(Note that $\psi_{a,b}$ is invertible near $\pm\infty$ and $P_{a,b}$ admits a series expansion near $0$.)

Then Problem 3 is equivalent to the following one (cf. \cite{B3} and the main result of \cite{Co}):}
\begin{P3'}
Is it true that a composition $P_{a_1,b_1}\circ P_{a_2,b_2}\circ\cdots \circ P_{a_n,b_n}$ such that all vectors $(a_i,b_i)$ and $(a_{i+1}, b_{i+1})$, $1\le i\le n-1$, $n\in\mathbb N$, are noncollinear is never the identity map?
\end{P3'}
\end{R}
\subsection{Lipschitz Embeddings of Universal Curves}
By $\|\cdot\|_2$ we denote the Euclidean norm on $\mathbb R^2$. Let $K\Subset\mathbb R^2$ be a compact subset containing the origin. A continuous map $F:K\rightarrow\mathbb R^2$ is called a {\em Lipschitz embedding} of $K$ if there exist constants $c_1,c_2>0$ such that
\begin{equation}\label{embed}
c_1\|x-y\|_2\le\|F(x)-F(y)\|_2\le c_2\|x-y\|_2\quad\text{for all}\quad x,y\in K.
\end{equation}
The set of Lipschitz embeddings of $K$ sending $0$ to $0$ will be denoted by ${\mathcal E\mathcal L}_0(K)$. Clearly if $F\in {\mathcal E\mathcal L}_0(K)$, then $F^{-1}$ exists and belongs to
${\mathcal E\mathcal L}_0(F(K))$. By $I\in {\mathcal E\mathcal L}_0(K)$ we denote the identity map. One easily shows that if $F\in {\mathcal E\mathcal L}_0(K)$ satisfies \eqref{embed}, then $F+\lambda I\in {\mathcal E\mathcal L}_0(K)$ for all $\lambda\in (-c_1,c_1)$. 
\begin{Th}\label{embed1}
Let $\Gamma\Subset\mathbb R^2$ be a universal curve and $F\in {\mathcal E\mathcal L}_0(\Gamma)$ satisfy \eqref{embed} for some $c_1,c_2>0$. Then there is an at most countable subset $S$ of the open interval $(-c_1,c_1)$ such that for all $\lambda\in (-c_1,c_1)\setminus S$ the Lipschitz triangulable curves $(F+\lambda I)(\Gamma)$ are universal.
\end{Th}
In Remark \ref{rem1} we give an explicit description of set $S$.\smallskip

Let $\mathcal L_0(K)$ be the Banach space of Lipschitz maps $F: K\rightarrow\mathbb R^2$ such that $F(0)=0$ equipped with norm
\[
\|F\|:=\sup_{x\ne y}\frac{\|F(x)-F(y)\|_2}{\|x-y\|_2}.
\]
One easily shows that if $F\in {\mathcal E\mathcal L}_0(K)$ satisfies \eqref{embed}, then $F+G\in {\mathcal E\mathcal L}_0(K)$ for all $G\in \mathcal L_0(K)$ such that $\|G\|< c_1$. In particular, ${\mathcal E\mathcal L}_0(K)$ is a nonempty open subset of $\mathcal L_0(K)$.

Suppose $\Gamma\Subset\mathbb R^2$ is a universal curve. Let $U(\Gamma)\subset\mathcal E\mathcal L_0(\Gamma)$ be the subset of Lipschitz embeddings $F:\Gamma\hookrightarrow\mathbb R^2$, $F(0)=0$, such that curves $F(\Gamma)$ are universal. We equip $U(\Gamma)$ with topology induced from $\mathcal E\mathcal L_0(\Gamma)$. The following result shows that $U(\Gamma)$ is a ``massive'' dense subset of $\mathcal E\mathcal L_0(\Gamma)$.
\begin{Th}\label{category}
There exists an at most countable family of closed nowhere dense subsets $S_i\subset\mathcal E\mathcal L_0(\Gamma)$ such that $U(\Gamma)=\mathcal E\mathcal L_0(\Gamma)\setminus\bigl(\cup_i\, S_i\bigr)$.
\end{Th}
In view of this result, the following problem seems to be plausible:
\begin{P4}
Let $\Gamma\Subset\mathbb R^2$ be a connected Lipschitz triangulable curve containing the origin. Is it true that there exists an $F\in\mathcal E\mathcal L_0(\Gamma)$ such that curve $F(\Gamma)$ is universal?
\end{P4}

\subsection{Approximation of Lipschitz Triangulable Curves by Universal Ones}
In what follows, $\mathcal H^1$ denotes the Hausdorff $1$-measure on $\mathbb R^2$ and $d_{\mathcal H}$ the Hausdorff metric on the set of compact subsets of $\mathbb R^2$. 
Let $\Gamma\Subset\mathbb R^2$ be a Lipschitz triangulable curve containing the origin. We say that $x\in\Gamma$ has {\em order} $k$ (written ${\rm ord}(x)=k$) if for all sufficiently small open disks $D_x$ centered at $x$, $\Gamma\setminus\{x\}$ has $k$ connected components in $D_x\setminus\{x\}$. (Note that since $\Gamma$ is triangulable, ${\rm ord}(x)$ is correctly defined.) By $S_{\Gamma}\subset\Gamma$ we denote the (finite) set of points of order $\ne 2$. Our main result is 
\begin{Th}\label{te3}
Suppose $\Gamma$ is piecewise linear. Then there exists a sequence $\{\Gamma_i\}_{i\in\mathbb N}$ of Lipschitz triangulable curves containing the origin such that 
\begin{itemize}
\item[(a)]
Each $\Gamma_i$ is a universal curve with $\pi_1(\Gamma_i)\cong\pi_1(\Gamma)$;\smallskip
\item[(b)]
${\rm ord}(x)\le 3$ for all $x\in\Gamma_i$, $i\in\mathbb N$;\smallskip
\item[(c)]
Each $\Gamma_i\setminus S_{\Gamma_i}$ is a closed $C^\infty$ submanifold of $\mathbb R^2\setminus S_{\Gamma_i}$;\smallskip
\item[(d)]
\[
\lim_{i\rightarrow\infty}\mathcal H^1(\Gamma_i)=\mathcal H^1(\Gamma)\quad\text{and}\quad \lim_{i\rightarrow\infty}d_{\mathcal H}(\Gamma_i,\Gamma)=0.
\]
\end{itemize}
\end{Th}
If (d) is valid for a sequence of curves $\{\Gamma_i\}_{i\in\mathbb N}$ and a curve $\Gamma$, then we say that $\{\Gamma_i\}_{i\in\mathbb N}$ converges to $\Gamma$ with respect to ${\mathcal H^1}$-measure and metric $d_{\mathcal H}$.
Since each Lipschitz triangulable curve can be approximated with respect to ${\mathcal H^1}$-measure and metric $d_{\mathcal H}$ by a sequence of piecewise linear curves (because each Lipschitz arc $[0,1]\rightarrow\mathbb R^2$ admits such approximation), as an immediate corollary of Theorem \ref{te3} we obtain:
\begin{C}\label{cor4}
Suppose $\Gamma\Subset\mathbb R^2$ is a Lipschitz triangulable curve containing the origin. Then there exists a sequence $\{\Gamma_i\}_{i\in\mathbb N}$ of universal curves converging to $\Gamma$ with respect to ${\mathcal H^1}$-measure and metric $d_{\mathcal H}$ such that each $\Gamma_i\setminus S_{\Gamma_i}$ is a closed $C^\infty$ submanifold of $\mathbb R^2\setminus S_{\Gamma_i}$.
\end{C}

The following problem is related to the composition conjecture for Abel equations with polynomial coefficients (see the Introduction).
\begin{P5}
Do there exist universal curves $\Gamma\Subset \mathbb R^2$ with $\pi_1(\Gamma)\cong F_m$, $m\ge 2$, which are images of analytic maps $[0,T]\rightarrow\mathbb R^2$? 
The same question is for images of polynomial maps.
\end{P5}
Regarding to this problem we mention that universal centers of Abel equations with real analytic coefficients can be easily described in terms of vanishing of finitely many moments in the coefficients of the equations, see the algorithm in the Introduction of \cite{B7}.
\subsection{Remarks on Nonuniversal Curves}
While the structure of universal centers of equation \eqref{eq1} is well understood, nonuniversal centers of this equation are of obscure nature requiring further scrutiny.
If equation \eqref{eq1} has a nonuniversal center on $[0,T]$ and the image $\Gamma_{\tilde A}$ of the corresponding map $\tilde A$ is Lipschitz triangulable, then $\Gamma_{\tilde A}$ is a nonuniversal curve; in particular, homomorphism $\hat P_{\Gamma_{\tilde A}}:\pi_1(\Gamma_{\tilde A})\rightarrow G_c[[r]]$ (see \eqref{compos}) has a nontrivial kernel (enclosing the minimal normal subgroup containing the element $[\tilde A]\in\pi_{1}(\Gamma_{\tilde A})$ represented by path $\tilde A$).
 This implies that for every $b=(b_1,b_2)\in\mathcal A_0(\Gamma_{\tilde A})$ such that the closed path $\tilde B:=\int_0^{\cdot}b(s)\,ds :[0,T]\rightarrow \Gamma_{\tilde A}$ is homotopic to a nontrivial element of ${\rm Ker}\, \hat P_{\Gamma_{\tilde A}}$ the corresponding Abel equation $\frac{dv}{dt}=b_1 v^2+b_2 v^3$ has a nonuniversal center on $[0,T]$.  Abel equations with nonuniversal centers can be obtained, e.g., by the application of the Cherkas transformation to certain elements of the Lotka-Volterra or Darboux  components of the set of centers of planar polynomial vector fields of degree 2, see \cite[Sect.~2.5.3]{Bl}. In these case coefficients $(a_1,a_2)$ of the obtained  Abel equations are trigonometric polynomials of degrees $\le 4$; hence the  corresponding curves $\Gamma_{\tilde A}$ are
Lipschitz triangulable (and nonuniversal). We mention also the following simple property: if $\Gamma\subset\tilde\Gamma\Subset\mathbb R^2$ is a pair of 
connected Lipschitz triangulable curves containing the origin and $\Gamma$ is nonuniversal, then $\tilde\Gamma$ is nonuniversal as well.

\sect{Proofs}
\begin{proof}[Proof of Theorem \ref{te1}]
First, assume that $\Gamma$ is universal. Suppose $\hat P(g)=1\in G_c[[r]]$ for some $g\in\pi_1(\Gamma)$. Let
$a_g\in\mathcal A_0(\Gamma)$ be such that $\Psi_\Gamma(a_g)=g$. Then due to \eqref{compos}, $P(a_g)=1$, i.e. equation \eqref{eq1} corresponding to $a_g$ has center on $[0,T]$. Since $\Gamma$ is universal, this center is universal.
Then by \cite[Th.\,1.14]{B2}, the closed Lipschitz path $\tilde A_g:=\int_0^{\cdot} a_g(s)\,ds : [0,T]\rightarrow\Gamma$ is contractible, that is
$g=\Psi_\Gamma(a_g)=1\in\pi_1(\Gamma)$. This shows that $\hat P_\Gamma$ is a monomorphism.

Conversely, suppose that $\hat P_\Gamma$ is a monomorphism. Let $a\in\mathcal A_0(\Gamma)$ be such that the corresponding equation \eqref{eq1} has center on $[0,T]$. Due to \eqref{compos} and our assumption this implies that $\Psi_\Gamma(a)=1\in\pi_1(\Gamma)$. Then the closed Lipschitz path $\tilde A: [0,T]$ is contractible. In turn, by the results of \cite{B2}, the corresponding center is universal. So $\Gamma$ is universal by our definition.
\end{proof}
\begin{proof}[Proof of Proposition \ref{prop3.1}]
If $\pi_1(\Gamma)=F_0$, then $\hat P_\Gamma: F_0\rightarrow G_c[[r]]$ is obviously a monomorphism; thus such curve $\Gamma$ is universal by Theorem \ref{te1}.

Next, if $\pi_1(\Gamma)\cong F_1$ but $\Gamma$ is not universal, then due to Theorem \ref{te1}
$\hat P_\Gamma$ maps $\pi_1(\Gamma)$ to $1\in G_c[[r]]$.  Let  $g$ be a generator of $\pi_1(\Gamma)$ and $a_g=(a_{g1},a_{g2})\in\mathcal A_0(\Gamma)$ be such that $\Psi(a_g)=g$. Then according to our hypothesis and \eqref{compos}, equation \eqref{eq1} corresponding to $a_g$ has center on $[0,T]$. In particular, from \eqref{eq2} with $i=3$ (see also notation \eqref{tildea}) we obtain
\begin{equation}\label{equ3.1}
3\cdot\int_0^T \tilde a_{g1}(s)\cdot a_{g2}(s)\, ds+2\cdot\int_0^T \tilde a_{g2}(s)\cdot a_{g1}(s)\, ds=0.
\end{equation}
Also, since $\tilde a_{gi}(T)=0$, $i=1,2$, we have
\[
\int_0^T \tilde a_{g1}(s)\cdot a_{g2}(s)\, ds+\int_0^T \tilde a_{g2}(s)\cdot a_{g1}(s)\, ds=\bigl(\tilde a_1(t)\cdot \tilde a_2(t)\bigr)|_0^T=0.
\]
Thus \eqref{equ3.1} is equivalent to
\begin{equation}\label{eq3.2}
\int_0^T \tilde a_{g1}(s)\cdot a_{g2}(s)\, ds=0.
\end{equation}
The latter can be rewritten as the contour integral of $1$-form $xdy$ over the closed path $\hat A_g:=\int_0^{\cdot} a_g(s)\, ds : [0,T]\rightarrow\Gamma$:
\[
\int_{\hat A_g}xdy=0.
\]
Since $\pi_1(\Gamma)\cong\mathbb Z$ and $\Gamma$ is Lipschitz triangulable, by the Jordan theorem $\mathbb R^2\setminus\Gamma$ contains only one bounded component $D_\Gamma$ homeomorphic to the open unit disk. In turn, by our definition $\hat A_g$ represents the generator $g\in\pi_1(\Gamma)$ and so by the Green formula we obtain
\[
0=\left|\int_{\hat A_g}xdy\right|=\left|\int_{D_\Gamma} dxdy\right|=Area(D)\ne 0,
\]
a contradiction which shows that $\hat P_{\Gamma}(g)\ne 1$, i.e. by Theorem \ref{te1} $\Gamma$ is a universal curve.
\end{proof}
\begin{proof}[Proof of Proposition \ref{prop3.2}]
Let $\Gamma\Subset\mathcal L$ be a curve satisfying conditions of the proposition. According to Theorem \ref{te1} we have to prove that homomorphism $\hat P_\Gamma:\pi_1(\Gamma)\rightarrow G_c[[r]]$ is injective. Consider
elements $a_i:=A_1'\in\mathcal A$, $i=1,2$ (here $'$ stands for the derivative of a path). According to \cite[Th.~2.4, Ex.~2.5]{B4} the first return maps $P(a_i)\in G_c[[r]]$, $i=1,2$,  generate a subgroup isomorphic to $F_2$. Further, each element $g\in\pi_1(\Gamma)$ is represented by a path $A_g: [0,1]\rightarrow\Gamma$ which is the product (in the sense of algebraic topology) of finitely many paths $A_i$ and their inverses. Assume that 
$\hat P_\Gamma(g)=1\in G_c[[r]]$. Then due to \eqref{compos} $P(A_g')=1\in G_c[[r]]$. Thus, the Abel equation \eqref{eq1} corresponding to the pair of coefficients $A_g'$ has center on $[0,1]$. But by our definition $A_g'$ is a finite $*$-product of elements $a_i$ and their inverses. Then due to \cite[Th.~2.4]{B4} this center is universal. This implies, see \cite{B2}, that path $A_g$ is closed and contractible inside its image. But $A_g$ has image in $\Gamma$; hence $A_g$ is contractible in $\Gamma$. This is equivalent to $g=1\in\pi_1(\Gamma)$. Therefore $\hat P_\Gamma$ is a monomorphism and by Theorem \ref{te1}, $\Gamma$ is a universal curve.
\end{proof}
\begin{proof}[Proof of Theorem \ref{embed1}]
Let $\ell_g: [0,T]\rightarrow\Gamma$ be a closed Lipschitz path representing an element $g\in\pi_1(\Gamma)$. Then for each $\lambda\in (-c_1,c_1)$ the countable family of paths $\{\ell_{g\lambda}\}_{g\in\pi_1(\Gamma)}$, $\ell_{g\lambda}:=(F+\lambda I)(\ell_g)$, represents all elements of the fundamental group $\pi_1(\Gamma_\lambda)\, (\cong\pi_1(\Gamma))$, $\Gamma_\lambda:=(F+\lambda I)(\Gamma)$. According to formula \eqref{eq2} the first return map of the Abel equation corresponding to the pair of coefficients $\ell_{g\lambda}'\in\mathcal A$ has the form
\[
P(\ell_{g\lambda}')(r)=r+\sum_{j=1}^\infty c_{gj}(\lambda)r^{j+1},
\]
where $c_{gj}$ is a real polynomial in $\lambda$ of degree at most $j$\, $(1\le j<\infty)$. 

Since $\Gamma$ is universal, the first return map $P(\ell_g')\subset G_c[[r]]$ is not the identity map (see Theorem \ref{te1}).  Thus there exists $j_g\in\mathbb N$ such that $c_{gj_g}(0)\ne 0$.
This implies that $c_{gj_g}$ is a not identically zero polynomial in $\lambda$ of degree at most $j_g$. By $S_{j_g}$ we denote the set of its zeros in $(-c_1,c_1)$. Clearly, ${\rm card}\, S_{j_g}\le j_g$. We define
\[
S:=\bigcup_{g\in \pi_1(\Gamma)}\, S_{j_g}.
\]
Since $\pi_1(\Gamma)$ is countable, the set $S\subset (-c_1,c_1)$ is at most countable. By definition, for each $t\in (-c_1,c_1)\setminus S$  and $g\in \pi_1(\Gamma)$ we have $P(\ell_{g\lambda}')\ne {\rm id}$.
This shows that the kernel of the corresponding homomorphism $\hat{P}_{\Gamma_\lambda}:\pi_1(\Gamma_\lambda)\rightarrow G_c[[r]]$ (see \eqref{compos}) is trivial. Therefore by Theorem \ref{te1},
$\Gamma_\lambda$ is a universal curve.

The proof of the theorem is complete.
\end{proof}
\begin{R}\label{rem1}
{\rm It is easily seen that coefficients of all polynomials $c_{gj}$ belong to the minimal numerical field $\mathbb F\subset\mathbb R$ containing all iterated integrals in $\ell_g'$ and $\ell_{g0}'$. Therefore to make sure that all $c_{gj}(\lambda_0)\ne 0$ it suffices to choose $\lambda_0$ being a transcendental number over $\mathbb F$. Since $\mathbb F$ is a countable set, its algebraic closure ${\rm cl}_{a}(\mathbb F)$ is countable as well. Thus we may define the required set $S$ of the theorem as $S:=(-c_1,c_1)\cap {\rm cl}_{a}(\mathbb F)$.}
\end{R}
\begin{proof}[Proof of Theorem \ref{category}]
We retain notation of the proof of Theorem \ref{embed1}. For each $\ell_g\in [0,T]\rightarrow\Gamma$, $g\in\pi_1(\Gamma)$, and $G\in\mathcal L_0(\Gamma)$ consider the first return map 
\[
P\bigl((G\circ\ell_g)'\bigr)=r+\sum_{j=1}^\infty c_{gj}(G)r^{j+1}.
\]
Due to equation \eqref{eq2} and the Rademacher theorem (on a.e. differentiability of Lipschitz maps), the coefficients $c_{gj}$ are real polynomials of degrees at most $j$ on the Banach space $\mathcal L_0(\Gamma)$. Since $\Gamma$ is universal, $P\bigl((I\circ\ell_g)'\bigr)\ne {\rm id}$. Therefore there exists $j_g\in\mathbb N$ such that $c_{gj_g}\not\equiv 0$. In turn, the set of zeros ${\mathcal Z}_{gj_g}$ of $c_{gj_g}$ is a closed nowhere dense subset of $\mathcal L_0(\Gamma)$. We have
\[
U(\Gamma)=\mathcal E\mathcal L_0(\Gamma)\setminus\left(\bigcup_{g\in\pi_1(\Gamma)}\left(\bigcap_{\{j\, :\, c_{gj}\not\equiv 0\}}\mathcal Z_{gj}\right)\right).
\]
Since $\pi_1(\Gamma)$ is countable and $\mathcal E\mathcal L_0(\Gamma)$ is an open subset of $\mathcal L_0(\Gamma)$, the previous identity yields the required statement.
\end{proof}
\begin{proof}[Proof of Theorem \ref{te3}]
The proof consists of two parts. In the first part we approximate $\Gamma$ by a sequence of piecewise linear curves satisfying all the above conditions but (c). Then we ``smooth the corners'' of curves of the constructed sequence to get the required sequence $\{\Gamma_i\}$.\medskip

{\bf I.} We consider $\Gamma$ as a connected finite plane graph whose vertex set contains $S_{\Gamma}$ and the origin. By definition, each vertex not in $S_{\Gamma}$ has degree $2$.
First, let us show how to reduce the approximation problem to certain plane graphs whose vertices have degrees $\le 3$.\smallskip

Making a rotation of $\Gamma$ about $0\in\mathbb R^2$ by a small angle, without loss of generality we may assume that all edges of $\Gamma$ are not parallel to the $x$-axis. Suppose $v\in S_{\Gamma}$ has degree $\ge 3$. Let $E(v)$ be the set of edges of $\Gamma$ emanating from $v$. Let $\ell_{\pm}:=\{(x,y)\in\mathbb R^2\,:\, y=v_{\pm}\, , {\rm sgn}(v_{\pm}-v_{y})=\pm 1 \}$ be two lines parallel to the $x$-axis with $y$ coordinates $v_+, v_-$ larger and smaller than that of $v$ (denoted by $v_{y}$). For sufficiently small $|v_{\pm}-v_y|$ the union of these lines intersect all edges in $E(v)$. Let $I_{\pm}\subset \ell_{\pm}$ be closed intervals with endpoints in $E(v)$ whose union contains all points in $(\ell_{+}\sqcup\ell_{-})\cap E(v)$. Consider a piecewise linear curve $\Gamma_v$ obtained from $\Gamma$ by removing all parts of edges of $E(v)$ between $\ell_{-}$ and $\ell_{+}$ except for one part for each choice of sign $+$ or $-$ and then adding intervals $I_{+}$ and $I_{-}$ instead (cf. Figure 3).
\begin{center}
\includegraphics[scale=0.8]{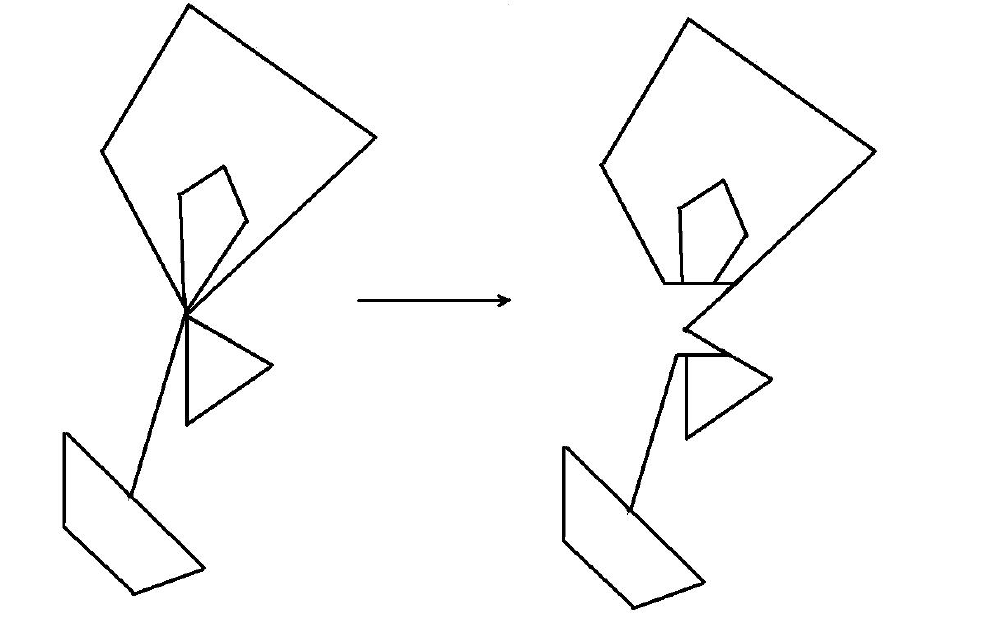} 

Figure 3. Example of curves $\Gamma$ and $\Gamma_v$. 
\end{center}

One can easily check that (for all sufficiently small $|v_{\pm}-v_y|$), $\pi_1(\Gamma_v)\cong\pi_1(\Gamma)$.  (Indeed, in this case according to the Seifert--van Kampen theorem the fundamental group of the union $U$ of $\Gamma$ and of the minimal convex set containing  $I_{+}, I_{-}$ and all parts of edges of $E(v)$ between $\ell_{-}$ and $\ell_{+}$ is isomorphic to $\pi_1(\Gamma)$. Also, $\Gamma_v$ is the deformation retract of $U$, i.e., $\pi_1(\Gamma_v)\cong\pi_1(U)\cong\pi_1(\Gamma)$ as required.) Moreover,  in a neighbourhood of $v$ all vertices of $\Gamma_v$ have degrees $\le 3$ and if a vertex there has degree $3$ then two of three edges adjacent to it lie on a straight line; all other vertices of $\Gamma_v$ are the same as for $\Gamma$, and the family of such curves $\Gamma_v$ converges to $\Gamma$ with respect to $\mathcal H^1$-measure and metric $d_{\mathcal H}$ (cf. (d) of the theorem) as $|v_{\pm}-v_y|$ and the above angle of rotation of $\Gamma$ about the origin tend to $0$. 

Applying subsequently a similar procedure to other vertices of $\Gamma$ with degrees $\ge 3$, after a finite number of steps we obtain a piecewise linear curve of the same homotopy type as $\Gamma$ with vertices of degrees $\le 3$ satisfying the same property as $\Gamma_v$ in a neighbourhood of $v$.  Moreover, there exists a sequence of such curves converging to $\Gamma$ with respect to $\mathcal H^1$-measure and metric $d_{\mathcal H}$.

Thus, from now on without loss of generality we will assume that {\em all vertices of $\Gamma$ have degrees $\le 3$, and if a vertex has degree $3$, then  two of three edges adjacent to it lie on a straight line}.

As a next step, we construct a rectangular piecewise linear curve $\Gamma_r$ by replacing edges of $\Gamma$ by unions of cathetii of right triangles with pairwise nonintersecting interiors of their sides and with hypotenuses lying on edges of $\Gamma$ (see Figure 4 below).  Also, we transfer to $\Gamma_r$ edges of $\Gamma$ parallel to one of the coordinate axes.  To show that the required $\Gamma_r$ exists, it suffices to construct the corresponding right triangles locally in neighbourhoods of vertices of $\Gamma$. This is possible since, due to the above property of $\Gamma$, each closed quadrant with respect to the coordinate axes centered at a vertex of $\Gamma$ contains in a small open neighbourhood of this vertex no more than $2$ edges of $\Gamma$. 
\begin{center}
\includegraphics[scale=1.16]{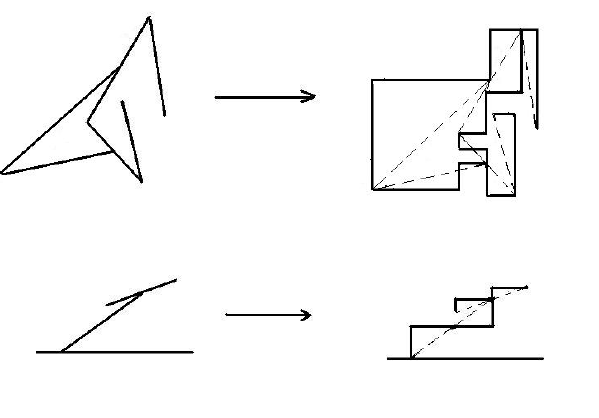}

Figure 4. Examples of curves $\Gamma$ and $\Gamma_r$.\smallskip
\end{center}

Since in a right triangle the union of cathetii can be isotopically deformed to the hypotenuse, the curve $\Gamma_r$ can be isotopically deformed to $\Gamma$ and, in particular, $\pi_1(\Gamma_r)\cong\pi_1(\Gamma)$. We make use of a specific isotopy between $\Gamma_r$ and $\Gamma$ defined in each right triangle $T$ with hypotenuse lying on an edge of $\Gamma$ as follows. Using a suitable affine transformation $L$ we map $T$ to the right triangle with coordinates $(0,0)$, $(a,0)$ and $(0,b)$, $a,b> 0$.
The hypotenuse in this triangle is given by equation $y=\frac ba (a-x)$ for $x\in [0,a]$. We deform the union of cathetii  $K_1:=\{(x,y)\, :\, x\in [0,a],\,  y=0\}$ and $K_2:=\{(x,y)\, :\, x=0,\, y\in [0,b]\}$ along the direction of the line $y=\frac ba x$ so that at time $t\in [0,1]$ cathetus $K_1$ transfers to the line $y=\frac{tb (a-x)}{a(2-t)}$, $x\in [\frac a2 t, a]$, and $K_2$ to the line $x=\frac{ta (b-y)}{b(2-t)}$, $y\in [\frac b2 t,b]$, see Figure 5. Then we apply the inverse affine transformation $L^{-1}$ to place the deformed break line back inside the original triangle $T$.
\begin{center}
\includegraphics[scale=1.8]{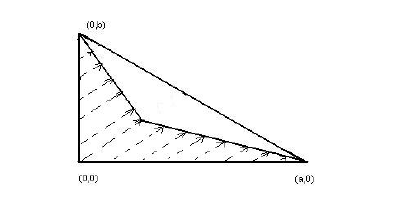}

Figure 5. Isotopy between the union of cathetii and the hypotenuse.\medskip
\end{center}

Without loss of generality we will assume that the vertex set of $\Gamma$ is a subset of the vertex set of $\Gamma_r$. By $\Phi : \Gamma_r\times [0,1]\rightarrow\mathbb R^2$ we denote the above described isotopy between $\Gamma_r$ and $\Gamma$ which deforms cathetii of each right triangle with hypotenuse an edge of $\Gamma$ toward the hypotenuse by the above formulas (modulo an affine transformation). We set $\Gamma_t:=\Phi(t,\Gamma_r)$.

Further, each {\em piecewise linear path} $\gamma: [0,T]\rightarrow \Gamma_r$ (i.e. a path composed of finitely many paths
moving along each edge with constant velocities) can be transformed using $\Phi$ to a piecewise linear path $\gamma_1: [0,T]\rightarrow\Gamma$.  We set $\gamma_t:=\Phi(\gamma,t)$, $t\in (0,1)$. Each $\gamma_t$ is a piecewise linear path in $\Gamma_t$.  Let  $\gamma_{t}'\in \mathcal A$ (see Section 2) be the velocity of path $\gamma_t$.
We require the following result.
\begin{Lm}\label{le4.2}
Each iterated integral $I_{i_{1},\dots, i_{k}}(\gamma_t')$ is a  real polynomial in $t$ of degree $\le k$.
\end{Lm}
\begin{proof}
By definition, $\gamma_t$ is product (in the sense of algebraic topology) of finitely many linear paths with images in edges of $\gamma_t$. In turn, $\gamma_t'$ is $*$-product of velocities of these linear paths. Due to formulas for iterated integrals of $*$-products, see \eqref{e2.2}, \eqref{e2.3}, it  suffices to prove the result for velocities of linear paths appearing as factors of $\gamma_t$ in the  above product. Then by the definition of $\gamma_t$ we must consider the following cases:

(1) $\gamma_t$ does not depend on $t$, this means that the image of $\gamma_1$ belongs to a vertical or a horizontal intervals. In this case the statement of the lemma is obvious.

(2) Image of $\gamma_t$ belongs to one of the intervals inside of a straight triangle with hypotenuse being an edge of $\Gamma$ (cf. Figure 5). Making an affine transformation of this triangle (which clearly does not affect the final result) without loss of generality we may assume that we are in position of Figure 5. Suppose for
definiteness that image of $\gamma_t$ belongs to the line $y=\frac{tb (a-x)}{a(2-t)}$, $x\in [\frac a2 t, a]$ (the lower line inside the triangle in Figure 5) and the endpoints of $\gamma_t$ have $x$-coordinates $\frac a2 t\le s_0<s_1\le a$. By the  definition of isotopy $\Phi$ in this case we have
\[
\gamma_t(s):=\Phi(\gamma_0(s),t)=\left( \frac{at}{2}+\frac{(2-t)\gamma_0(s)}{2}, \frac{bt}{2}-\frac{bt\gamma_0(s)}{2a} \right),\qquad s\in [0,T].
\]
Here $\gamma_0:[0,T]\rightarrow K_1\subset\mathbb R$ is a linear path with image in the cathetus $K_1$ and endpoints $s_0,s_1\in K_1$. Assuming that
$\gamma_0'=v_0\, (\in\mathbb R)$, we obtain
\[
\gamma_t'(s):=\left(\frac{(2-t)v_0}{2}, -\frac{bt v_0}{2a} \right),\qquad s\in [0,T].
\]
Since $\gamma_t'$ depends linearly on $t$, we have (see \eqref{eq2}) that  each iterated integral $I_{i_{1},\dots, i_{k}}(\gamma_t')$ is a  real polynomial in $t$ of degree at most $k$ as required.

All other choices of $\gamma_t$ are treated similarly.
\end{proof}

Using Lemma \ref{le4.2} we prove
\begin{Proposition}\label{prop4.3}
There exists an at most countable subset $Z\subset [0,1]$ such that for each $t\in [0,1]\setminus Z$ the piecewise linear curve $\Gamma_t$ is universal.
\end{Proposition}
\begin{proof}
Since $\Gamma_r:=\Gamma_0$ is a rectangular curve, it is universal, see \cite{B3}. Let  $\ell_1,\dots,\ell_m$ be closed rectangular paths in $\Gamma_r$ containing the origin representing generators of $\pi_1(\Gamma_r)\cong F_m$ (without loss of generality we assume that $m\ge 1$). Each element $g\in \pi_1(\Gamma_r)$ is represented by a closed rectangular path $\ell_g$ which is the finite product of paths $\ell_i$ or their inverses. Further, the isotopy $\Phi$ induces isomorphisms $\pi_1(\Gamma_t)\cong\pi_1(\Gamma_r)$. In particular, closed paths $(\ell_g)_t:[0,T]\rightarrow\Gamma_t$, $g\in\pi_1(\Gamma_r)$, are mutually distinct and represent all elements of $\pi_1(\Gamma_t)$.
According to Lemma \ref{le4.2} and formula \eqref{eq2} (containing expressions for coefficients of series expansions of the first return maps of Abel equations), 
\[
P((\ell_{g})_t')(r)=r+\sum_{j=1}^\infty c_{gj}(t)r^{j+1},
\]
where $c_{gj}$ is a real polynomial in $t$ of degree at most $j$\,  ($1\le j<\infty$).

Since $\Gamma_r$ is universal, the first return map $P(\ell_g')\subset G_c[[r]]$ is not the identity map (see Theorem \ref{te1}).  Thus there exists $j_g\in\mathbb N$ such that $c_{gj_g}(0)\ne 0$.
This implies that $c_{gj_g}$ is a not identically zero polynomial in $t$ of degree at most $j_g$. By $Z_{j_g}$ we denote the set of its zeros in $[0,1]$. Clearly, ${\rm card}\, Z_{j_g}\le j_g$. We define
\[
Z:=\bigcup_{g\in \pi_1(\Gamma_r)}\, Z_{j_g}.
\]
Since $\pi_1(\Gamma_r)$ is countable, the set $Z\subset [0,1]$ is at most countable. By definition, for each $t\in [0,1]\setminus Z$  and $g\in \pi_1(\Gamma_r)$ we have $P((\ell_{g})_t')\ne {\rm id}$.
This shows that the kernel of the corresponding homomorphism $\hat{P}_{\Gamma_t}:\pi_1(\Gamma_t)\rightarrow G_c[[r]]$ (see \eqref{compos}) is trivial. Therefore by Theorem \ref{te1},
$\Gamma_t$ is a universal curve.

The proof of the proposition is complete.
\end{proof}

Now, we choose a sequence $\{t_i\}_{i\in\mathbb N}\subset [0,1]\setminus Z$ converging to $1$. Then according to our construction and the previous proposition piecewise linear curves $\Gamma_{t_i}$ are universal, isotopic to $\Gamma$ and the sequence of such curves converges to $\Gamma$  as $t_i\rightarrow 1$ with respect to $\mathcal H^1$-measure and metric $d_{\mathcal H}$ (see (d) of Theorem \ref{te3}).  This finishes part {\bf I} of the proof.\smallskip

{\bf II.} In this part given a piecewise linear universal curve $\Gamma$ we approximate it (with respect to $\mathcal H^1$-measure and metric $d_{\mathcal H}$) by a sequence of isotopic universal curves having the same sets of points of order $\ne 2$ and $C^\infty$ outside of these sets.

As in part {\bf I}, we isotopically deform $\Gamma$ to smooth its corners. Specifically, 
let $\mathcal V_2$ be the set of vertices of $\Gamma$ of degree $2$ whose adjacent edges do not lie on a straight line. Fix a family of mutually disjoint open connected neighbourhoods $U_v$ of $v\in \mathcal V_2$ such that each $U_v$ does not contain other vertices but $v$. 
We deform simultaneously parts of $\Gamma$ in all $U_v$, $v\in\mathcal  V_2$,  preserving points of $\Gamma$ outside of these neighbourhoods. To this end we map each $U_v$  by a suitable affine transformation $L_v$ onto the angle $R_0:=\{(x,y)\in\mathbb R^2\, :\, y=|x|,\ x\in [-2,2]\}$ (see Figure 6). We construct an isotopy for this angle which preserves parts of the edges situated outside the open disk of radius $\sqrt 2$ centered at the origin and then  applying the inverse affine transformation $L_v^{-1}$ attach the deformed piece of the curve to $\Gamma$.

\begin{center}
\includegraphics[scale=1.4]{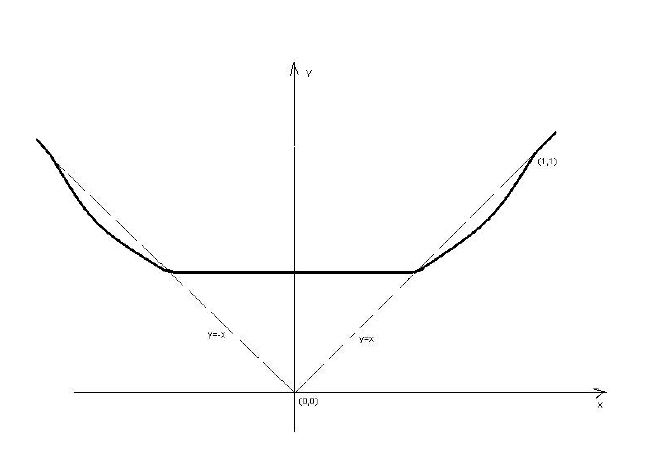}

Figure 6. Smoothing the corner of the angle formed by  the graph $y=|x|$.\smallskip
\end{center}
For this purpose consider an auxiliary $C^\infty$ function
\[
\rho(x):=\left\{
\begin{array}{ccc}
e^{-\frac 1x}&{\rm if}&x> 0\smallskip\\
0&{\rm if}&x\le 0.
\end{array}
\right.
\]
Then the isotopy $\phi : R_0\times [0,1]\rightarrow \mathbb R^2$ is given by the formula $\phi((x, y),t):=\left(x,th(\frac xt)\right)$\, ($y:=|x|$, $x\in \mathbb R$), where
\begin{equation}\label{cutoff}
h(x):=\frac 12+\frac{(|x|-\frac 12)\rho(2|x|-1)}{\rho(2|x|-1)+\rho(2-2|x|)}.
\end{equation}
(By the definition, $h$ is a $C^\infty$ function on $\mathbb R$ equals $\frac 12$ if $|x|\le\frac 12$ and $|x|$ if $|x|\ge 1$.)

By $\Phi: \Gamma\times [0,1]\rightarrow\mathbb R^2$ we denote the (above described) isotopy given by the formula $\Phi(z,t):=L_{v}^{-1}\left(\phi( L_v(z),t)\right)$, $(z,t)\in U_v\times [0,1]$, for all $v\in\mathcal V_2$.
(Without loss of generality we assume that $\Phi$ is correctly defined, i.e.,  that $U_v$ were chosen so small that for each $t\in [0,1]$ images $\Phi(U_v,t)$ for distinct $v$ are pairwise disjoint.) We set $\Gamma_t:=\Phi(\Gamma, t)$ and follow the lines of the proof given in step {\bf I}. Specifically,  for a closed piecewise linear path $\gamma: [0,T]\rightarrow \Gamma$ representing a nontrivial element of $\pi_1(\Gamma)$
we define $\gamma_t:=\Phi(\gamma,t): [0,T]\rightarrow\Gamma_t$. Then,  for this choice of $\gamma_t$, we require an analog of Lemma \ref{le4.2}:
\begin{Lm}\label{le4.4}
Each iterated integral $I_{i_{1},\dots, i_{k}}(\gamma_t')$ is a  real polynomial in $t$ of degree $\le k$.
\end{Lm}
\begin{proof}
Adding additional vertices, if necessary, without loss of generality we may assume that each edge of $\Gamma$ contains at most one vertex from $\mathcal V_2$.

Since iterated integrals are constant on the set of derivatives of closed paths representing the same element of $\pi_1(\Gamma)$, one can choose $\gamma$ so that it is a finite product of linear paths whose images are edges of $\Gamma$. Hence, due to formulas \eqref{e2.2} and \eqref{e2.3}, it suffices to prove the result for a linear path $\gamma$ whose image is an edge $e$ of $\Gamma$. If $e$ joins two vertices not in $\mathcal V_2$, then $\gamma_t=\gamma$ for all $t$ and the required result is obvious. For otherwise, making a suitable transformation $L_v$, without loss of generality we may assume that $\gamma$ is a linear map onto a subinterval of the line $y=x$ with endpoints $(0,0)$ and $(a,a)$ for some $a>2$. Applying anew formulas 
\eqref{e2.2} and \eqref{e2.3} to a proper factorization of $\gamma$, we see that it suffices to prove the lemma for $\gamma$ whose image is interval $[(0,0), (1,1)]$. 
In this case making a change of parameter we may assume that $\gamma(s):=(s,s)$, $s\in [0,1]$. Hence, $\gamma_t(s)=\left(s, th(\frac st)\right)$, $s\in [0,1]$. We present $\gamma_t$ as product of paths $\gamma_{1t}(s):=\gamma_t(ts)$, $s\in [0,1]$, and $\gamma_{2t}(s):=\gamma_t(t+(1-t)s)$, $s\in [0,1]$. Therefore due to formulas \eqref{e2.2}, \eqref{e2.3} it suffices to prove the result for iterated integrals of derivatives of paths $\gamma_{it}$, $i=1,2$. We have
\[
\gamma_{1t}'(s):=(t,th'(s)),\qquad s\in [0,1].
\]
Thus, see \eqref{eq2}, 
\[
I_{i_{1},\dots, i_{k}}(\gamma_{1t}')=t^k\cdot I_{i_{1},\dots, i_{k}}(\gamma_{11}')\, (=const\cdot t^k)
\]
is a polynomial in $t$ of degree at most $k$.

In turn,
\[
\gamma_{2t}'(s):=\left(1-t,(1-t)\cdot h'\left(1+\frac{1-t}{t}s\right)\right)=(1-t,1-t),\qquad s\in [0,1],
\]
since $h(x)=x$ for $x\ge 1$.
Hence,
\[
I_{i_{1},\dots, i_{k}}(\gamma_{2t}')=\frac{(1-t)^k}{k!}
\]
is a polynomial in $t$ of degree at most $k$ as well.

This completes the proof of the lemma.
\end{proof}

Repeating word-for-word the arguments of the proof of Proposition \ref{prop4.3} (now based on Lemma \ref{le4.4}) we get in our case
\begin{Proposition}\label{prop4.5}
There exists an at most countable subset $Z\subset [0,1]$ such that for each $t\in [0,1]\setminus Z$ the curve $\Gamma_t$ (which is $C^\infty$ outside the set $S_\Gamma$) is universal.
\end{Proposition}

Choosing a sequence $\{t_i\}_{i\in\mathbb N}\subset [0,1]\setminus Z$ converging to $0$ we obtain that curves $\Gamma_{t_i}$, having the same sets $S_{\Gamma_i}$ and $C^\infty$ outside of these sets, are universal, isotopic to $\Gamma$ and the sequence of such curves converges to $\Gamma$  as $t_i\rightarrow 0$ with respect to $\mathcal H^1$-measure and metric $d_{\mathcal H}$.  

This finishes the proof of part {\bf II} and hence of the theorem.
\end{proof}

\begin{R}\label{concluding}
{\rm According to our construction, the sequence of universal curves $\{\Gamma_i\}_{i\in\mathbb N}$ converging to $\Gamma$ has the property that for each $x\in S_{\Gamma_i}$ of order $3$ there is an open disk $D_x$ centered at $x$ such that $\Gamma_i\cap (D_x\setminus\{x\})$ is the union of three open intervals. Using the method of part {\bf II} of the proof of Theorem \ref{te3} one can additionally smooth curves $\Gamma_i$ to obtain new curves $\widetilde \Gamma_i$ satisfying conditions (a)--(d) of the theorem and such that}
for each $x\in S_{\widetilde \Gamma_i}$, ${\rm ord}(x)=3$, in an open disk $D_x$ centered at $x$,  $\widetilde \Gamma_i$ is the union of a connected $C^\infty$ curve and its tangent at $x$ having infinite order contact at $x$.
{\rm Moreover, one can easily show that such a curve $\widetilde\Gamma_i$ can be obtained as the image of a $C^\infty$ map $\tilde A: [0,T]\rightarrow\mathbb R^2$ which corresponds to an Abel equation \eqref{eq1} with $C^\infty$ coefficients. We leave the details to the reader.}
\end{R}


\begin{thebibliography}{}

\bibitem[AL]{AL}
M. A. M. Alwash and N. G. Lloyd, Non-autonomous equations related to polynomial two-dimensional systems,
Proc. R. Soc. Edinb. A {\bf 105} (1987), 129--152.

\bibitem[A1]{A1}
M.A.M. Alwash, On the Composition Conjectures, Electronic Journal of Differential Equations, Vol. 2003 (2003), No. 69, pp. 14.

\bibitem[A2]{A2}
 M.A.M. Alwash, The composition conjecture for Abel equation, Expo. Math. {\bf 27} (2009), 241--250.
 
 \bibitem[A3]{A3}
[M.A.M. Alwash, Polynomial differential equations with piecewise linear coefficients, Differ. Equ. Dyn. Syst. {\bf 19} (2011), no. 3, 267--281.
 
 \bibitem[A4]{A4}
M.A.M. Alwash, Composition conditions for two-dimensional polynomial systems, Diff. Eq. and Appl., Volume 5, Number 1, February 2013.

\bibitem[B1]{B1}
A. Brudnyi, An explicit expression for the first return map in the center problem, J. Differ. Equations  {\bf 206} (2004), 306--314.

\bibitem[B2]{B2}
A. Brudnyi, On the center problem for ordinary differential equations, Amer. J. Math. {\bf 128} (2) (2006), 419--451.

\bibitem[B3]{B6}
A. Brudnyi, Formal paths, iterated integrals and the center problem
for ordinary differential equations, Bull. Sci. math. {\bf 132} (2008), 455--485.

\bibitem[B4]{B3}
A. Brudnyi, Center problem for ODEs with coefficients generating the group of rectangular paths, C. R. Math. Acad. Sci. Soc. R. Can. {\bf 31} (2009), no. 2, 33--44.

\bibitem[B5]{B4}
A. Brudnyi, Free subgroups of the group of formal power series and the center problem for ODEs, C. R. Math. Acad. Sci. Soc. R. Can. {\bf 31}  (2009), no. 4, 97--106.

\bibitem[B6]{B5}
A. Brudnyi, Composition conditions for classes of analytic functions, Nonlinearity {\bf 25} (2012), 3197--3209.

\bibitem[B7]{B7}
A. Brudnyi, Moments finiteness problem and characterization of universal centers of ODEs with analytic coefficients, to appear.  arXiv:1305.4303.

\bibitem[BlRY]{BlRY} 
M. Blinov, N. Roytvarf and Y. Yomdin, Center and moment conditions for 
Abel 
equation with rational coefficients, Funct. Differ. Equ.  {\bf 10}  (2003),  no. 
1-2, 95--106.

\bibitem[BFY1]{BFY1}
 M. Briskin, J.-P. Francoise and Y. Yomdin, The Bautin ideal of the Abel Equation, Nonlinearity  {\bf 11} (1998), 41--53.

\bibitem[BFY2]{BFY2}
M. Briskin,  J.-P. Francoise and Y. Yomdin, Center conditions, compositions of polynomials and moments on
algebraic curves, Erg. Theory Dyn. Syst. {\bf 19} (1999), 1201–--1220.

\bibitem[Bl]{Bl}
M. Blinov, Center and composition conditions for Abel equation, Thesis. Weizmann Institute of Science, 2002.

\bibitem[BRY]{BRY}
M. Briskin, N. Roytvarf and Y. Yomdin,  Center conditions at infinity for Abel differential equations, Annals of Math. {\bf 172} (1) (2010), 437--483.

\bibitem[BY]{BY}
A. Brudnyi and Y. Yomdin, Tree composition condition and moments vanishing, Nonlinearity {\bf 23} (2010), 1651--1673.

\bibitem[C]{C}
C. Christopher, Abel equations: composition conjectures and the model problem, Bull. Lond. Math. Soc.
{\bf 32} (2000), 332-–338.

\bibitem[CGM1]{CGM1}
A. Cima, A. Gasull, and F. M\~{a}nosas, Centers for trigonometric Abel equations, Qual.
Theory Dyn. Syst. {\bf 11} (2012), 19--37.

\bibitem[CGM2]{CGM2}
A. Cima, A. Gasull, and F. M\~{a}nosas, A simple solution of some composition conjectures for Abel equations, J. Math. Anal. Appl. {\bf 398} (2013), 477–-486.

\bibitem[CGM3]{CGM3}
A. Cima, A. Gasull, and F. M\~{a}nosas, An explicit bound of the number of vanishing double moments forcing composition, J. Diff. Equations {\bf 255} (3) (2013), 339--350.


\bibitem[GGL]{GGL}
J. Gin\'e, M. Grau and J. Llibre, Universal centers and composition conditions,
Proc. London Math. Soc. {\bf 106} (3) (2013), 481--507.

\bibitem[Co]{Co}
 S.D. Cohen, The group of translations and positive rational powers is free, Quart. J. Math. Oxford (2)  {\bf 46} (1995), 21--93.

\bibitem[Hu]{Hu}
S.-T. Hu, Homotopy Theory, New York, 1959.

\bibitem[I]{I}
Yu. Il'yashenko, Centennial history of Hilbert's 16th problem, Bull. Amer. Math. Soc. (N.S.) {\bf 39} (3) (2002), 301--354.

\bibitem[MP]{MP} 
M. Muzychuk and F. Pakovich, Solution of the polynomial moment problem, Proc. Lond. Math. Soc. {\bf 99} (3) (2009), 633--657.

\bibitem[P1]{P1} 
F. Pakovich, On the polynomial moment problem, Math. Res. Lett., {\bf 10}, 
no. 2-3 (2003), 401--410.

\bibitem[PRY]{PRY}
F. Pakovich, N. Roytvarf and Y. Yomdin, Cauchy type integrals
of algebraic functions, Isr. J. Math., Vol. 144 (2004), 221--291.

\bibitem[P2]{P2}
F. Pakovich, On rational functions orthogonal to all powers of a given rational 
function on a curve, arXiv:0910.2105.






\end{thebibliography}
\end{document}